\documentclass[11pt]{article}

\usepackage{amsmath, amsthm}
\usepackage{amssymb}
\usepackage{eucal}
\usepackage{hyperref} 
\usepackage{mathrsfs}

\usepackage{enumerate}

\usepackage{xcolor}

\newtheorem{theorem}{Theorem}[section]
\newtheorem{lemma}[theorem]{Lemma}

\newtheorem{remark}[theorem]{Remark}

\newtheorem{definition}[theorem]{Definition}
\newtheorem{hypothesis}[theorem]{Hypothesis}

\DeclareSymbolFont{symbolsC}{U}{txsyc}{m}{n}
\DeclareMathSymbol{\notniFromTxfonts}{\mathrel}{symbolsC}{61}

\begin{document}
	
	\baselineskip 16pt

	\title{The conjugacy diameters of non-abelian finite $p$-groups with cyclic maximal subgroups}

	\author{Fawaz Aseeri\footnote{Corresponding author.}\\
		{\small Mathematics Department, Faculty of Sciences, Umm Al-Qura University,}\\ {\small Makkah 21955, Saudi Arabia}\\
		{\small E-mail:
			fiaseeri@uqu.edu.sa}\\ \\
		{Julian Kaspczyk\footnote{At the date of submission, this author was not anymore affiliated with the Technische Universit\"{a}t Dresden. However, the bulk of the work presented here was done when he still was a postdoc at this university.}}\\
		{\small Institut f\"{u}r Algebra, Fakult\"{a}t Mathematik, Technische Universit\"{a}t Dresden,}\\
		{\small 01069 Dresden, Germany}\\
		{\small E-mail: julian.kaspczyk@gmail.com}}
	
	\date{}
	\maketitle

	\begin{center}
	\textit{We dedicate this paper to our former supervisor Professor Benjamin Martin. \newline You have our deepest thanks.}
	\end{center}
	\begin{abstract}

Let $G$ be a group. A subset $S$ of $G$ is said to normally  generate $G$ if $G$ is the normal closure of $S$ in $G.$ In this case, any element of $G$ can be written as a product of conjugates of elements of $S$ and their inverses. If $g\in G$ and $S$ is a normally generating subset of $G,$ then we write $\| g\|_{S}$ for the length of a shortest word in $\mbox{Conj}_{G}(S^{\pm 1}):=\{h^{-1}sh | h\in G, s\in S \, \mbox{or} \, s{^{-1}}\in S  \}$ needed to express $g.$ For any normally generating subset $S$ of $G,$ we write $\|G\|_{S} =\mbox{sup}\{\|g\|_{S} \,|\,\, g\in G\}.$ Moreover, we write $\Delta(G)$ for the supremum of all $\|G\|_{S},$ where $S$ is a finite normally generating subset of $G,$ and we call $\Delta(G)$ the conjugacy diameter of $G.$ In this paper, we determine the conjugacy diameters of the semidihedral $2$-groups, the generalized quaternion groups and the modular $p$-groups. This is a natural step after the determination of the conjugacy diameters of dihedral groups, which were recently found by the first author (finite case) and by  K{\c e}dra,\,{Libman} and {Martin} (infinite case).
		
	\end{abstract}		
	

	\let\thefootnote\thefootnoteorig

	\section{Introduction}

	\subsection{Background on norms and boundedness} 
	Let $G$ be a group. A \textit{norm}  on $G$ is a function $\nu:G\longrightarrow [0,\infty)$ which satisfies the following axioms: 
 \hfill\\
\,\,\,\textcolor{white}{.......}(i) $\nu(g)=0 \iff g=1;$\\
\textcolor{white}{......}(ii) $\nu(g^{-1})=\nu(g)$ for all $g\in G;$ \\
\textcolor{white}{.....}(iii) $\nu(gh)\le \nu(g)+\nu(h)$ for all $g,h \in G.$\\
We call $\nu$ \textit{conjugation-invariant} if we also have\\
\textcolor{white}{.....}(iv) $\nu(g^{-1}hg)=\nu(h)\,\, \mbox{for all}\,\, g,h\in G.$\\

 The \textit{diameter} of a group $G$ with respect to a conjugation-invariant norm $\nu$ on $G$ is defined as  \[\mbox{diam}_\nu(G):=\sup\{\nu{(g)}| g\in G\}.\]  
A group $G$ is said to be \textit{bounded} if every conjugation-invariant norm on $G$ has a finite diameter.
 The concepts of conjugation-invariant norm and boundedness were introduced by Burago, Ivanov and Polterovich in \cite{Burgo} and they provided a number of applications to geometric group theory, Hamiltonian dynamics and finite groups. Since then, there has been a large interest in providing further applications, examples and analogues of these notions in geometry, group theory and topology, see \cite{vla3,Kedraa,fragg, MorTok,Ben, Mur} for example.

 An important source of conjugation-invariant norms are normally generating subsets of groups. Let $G$ be a group and $S \subseteq G.$ The \textit{normal closure of $S$} in $G,$ denoted by $\langle\langle S \rangle\rangle,$ is the subgroup of $G$ generated by all conjugates of elements of $S$. In other words, it is the smallest
normal subgroup of $G$ containing $S.$ We say that $S$ \textit{normally generates} $G$ if $G=\langle\langle S \rangle\rangle.$ In this case, any element of $G$ can be written as a product of elements of 	
\begin{equation}
\label{definitionofconjG}
\mbox{Conj}_{G}(S^{\pm 1}):
=\{h^{-1}sh | h\in G, s\in S \, \mbox{or} \, s{^{-1}}\in S  \}.
\end{equation}

	If $S$ normally generates $G,$ then the \textit{length}\,$\|g\|_{S}\in \mathbb{N}$ of $g\in G$ with respect to $S$ is defined to be the length of a shortest word in $\mbox{Conj}_{G}(S^{\pm 1})$ that is needed to express $g.$ In other words, 
\[\| g\|_{S}=\mbox{inf}\{n\in \mathbb{N} | g=s_{1}\cdots s_{n}  \,\,\mbox{for some}\,\,  s_{1},
\dots,s_{n}\in  \mathrm{Conj}_{G}(S^{\pm 1}) \}.\]

It is important to stress that $\| 1\|_S$, the length of the identity element $1$ of $G$ with respect to $S$, is $0$. \textit{The word norm}  $\parallel .\parallel_S\colon G\rightarrow [0,\infty),$   $g\mapsto \| g\|_S$ is a conjugation-invariant norm on $G.$ The \textit{diameter} of a group $G$ with respect to the word norm $\parallel .\parallel_S$ is    \[\|G\|_{S}:=\mbox{sup}\{\|g\|_{S}  | g\in G\}.\] 
For other examples of conjugation-invariant norms, see \cite{Kedraa,Burgo}.

If $G$ is a group and $S$ is a finite normally generating subset of $G,$ then $G$ is bounded if and only if $\|G\|_{S}$ is finite (see \cite[Corollary $2.5$]{Ben}). Therefore, word norms with respect to finite normally generating subsets and their diameters are an important tool to study boundedness of groups.

 Moreover, word norms are used in the study of several refinements of the concept of bounded groups. To mention these refinements, we introduce some notation.  For any group $G$ and any $n\geq 1,$ let 
 \[\Gamma_{n}(G):=\{S\subseteq G\,\,|\,\, |S|\leq n\,\, \mbox{and} \,\,S \,\,\mbox{normally generates} \,\,G\},\]
\[\Gamma(G):=\{S\subseteq G\,\,|\,\, |S|< \infty\,\, \mbox{and} \,\,S \,\,\mbox{normally generates} \,\,G\}.\]
Set
 \begin{gather*}
\Delta_{n}(G):=\sup\{{\|G\|_{S}}\,\,|S\in \Gamma_{n}(G) \},
\\
\Delta(G):=\sup\{ {\|G\|_{S}}\,\,| S\in \Gamma(G) \}.
\end{gather*}

For any group $G,$ $\Delta(G)$ is said to be \textit{conjugacy diameter of} $G.$ The group $G$ is called \textit{strongly bounded} if $\Delta_{n}(G)<\infty$ for all $n\in\mathbb{N}$ (see \cite[Definition $1.1$]{Ben}).  The group $G$ is called \textit{uniformly bounded} if $\Delta(G)<\infty$  (see \cite[Definition $1.1$]{Ben}).

\subsection{Motivation and statements of results}

As we have seen in the previous section, word norms and their diameters emerge in the study of boundedness. Since any finite group is known to be uniformly bounded, the study of boundedness mainly focuses on infinite groups. However, word norms and their diameters are also of interest in finite group theory. For example, conjugacy diameters of finite groups can be used to study conjugacy class sizes (see \cite[Proposition 7.1]{Ben}). Recently, conjugacy diameters were studied for several classes of finite groups. For example, {K{\c e}dra},\,{Libman} and {Martin} showed that $\Delta(PSL(n,q))\leq 12(n-1)$
for any $n\geq 3$ and any prime power $q$ (see \cite[Example 7.2]{Ben}). Also, {Libman} and {Tarry} proved that $\Delta(S_n)= n-1$ for any $n\geq 2$ (see \cite[Thereom 1.2]{Tarry}) and that if $G$ is a nonabelian group of order $pq$, where $p$ and $q$ are prime numbers with $p < q$, then $\Delta(G) = \mathrm{max} \lbrace \frac{p-1}{2},2 \rbrace$ (see \cite[Thereom 1.1]{Tarry}). The first author determined the conjugacy diameters of finite dihedral groups:

\begin{theorem} \cite[Theorem 6.0.2]{Fawaz} 
\label{FawazResult}
Let $n\geq 3$ be a natural number and $G:=D_{2n}=\langle a,b|a^{n}=1=b^{2},bab=a^{-1}\rangle$ be the dihedral group of order $2n.$ Then 
\label{diemaeterofdihedralgroup}
 \[\Delta(G)=
 \left\{
  \begin{array}{@{}ll@{}}
  2 \,\,\,\mathrm{if}\, \,n\geq 3\,\mathrm{and}\,\,$n$\,\,\mathrm{odd},\\
  
    2 \,\,\,\mathrm{if}\, \,n= 4,\\\
   3   \,\,\,\mathrm{if}\, \,n\geq 6\,\mathrm{and}\,\,$n$\,\,\mathrm{even}.\,\,\,      \\
             \end{array}
 \right. \] 
\end{theorem}
\noindent For the case of the infinite dihedral group $D_\infty= \langle a,b|b^{2}=1, b a b=a^{-1}\rangle$, we have 
$\Delta(D_\infty)\leq 4$ (see \cite[Example\,2.8]{Ben}). \,

The goal of this paper is to take Theorem \ref{FawazResult} further. For any natural number $n\geq 3,$ the dihedral group $D_{2n}$ is non-abelian and has a cyclic maximal subgroup. Thus, one way to extend Theorem \ref{FawazResult} would be to prove further results about conjugacy diameters of non-abelian finite groups with cyclic maximal subgroups. Instead of looking at arbitrary non-abelian finite groups with cyclic maximal subgroups, let us restrict our attention to non-abelian finite $p$-groups with cyclic maximal subgroups.

 \begin{definition}  Let $n$ be a natural number. 
    \begin{enumerate} [(i)]
    \item  For $n\geq 4,$ define  
\[SD_{n}:=\langle a,b|a^{2^{n-1}}=1=b^{2},a^{b}=a^{-1+2^{n-2}}\rangle.\]
We call $SD_{n}$ the \textit{semidihedral} group of order $2^{n}.$ \\

  \item
   For $n\geq 3,$ define 

\[Q_{n}:=\langle a,b|a^{2^{n-1}}=1, b^{2}=a^{2^{n-2}} ,a^{b}=a^{-1}\rangle.\]

We call $Q_{n}$ the \textit{(generalized) quaternion} group of order $2^{n}.$\\

  \item
  
  Let $p$ be a prime number and assume that $n\geq 4$ if $p=2$ and $n\geq 3$ if $p$ is odd. Define

\[M_{n}(p):=\langle a,b|a^{p^{n-1}}=1=b^{p},a^{b}=a^{1+p^{n-2}}\rangle.\]

We call $M_{n}(p)$ the \textit{modular} $p$-group of order $p^{n}.$
  
   \end{enumerate}
   
 \end{definition}

The non-abelian finite $p$-groups with cyclic maximal subgroups are fully classified by the following theorem.

\begin{theorem} (see \cite[Chapter 5, Theorem 4.4]{Daniel}) Let $P$ be a non-abelian $p$-group of order $p^{n}$ which contains a cyclic subgroup $H$ of order $p^{n-1}.$\,Then 
 \begin{enumerate} [(i)]
    \item If $n=3$ and $p=2,$ then $P$ is isomorphic to $D_8$ or $Q_3.$
  
          \item If $n>3$ and $p=2,$ then $P$ is isomorphic to $M_n(2), D_{2^{n}}, Q_{n},$ or $SD_{n}.$
\item If $p$ is odd, then $P$ is isomorphic to $M_n(p).$
                              
        \end{enumerate}
\end{theorem}

 No two of the groups $D_{2^{n}}, Q_{n}, SD_{n},$ or $M_{n}(p)$ are isomorphic to each other  
(see\,\cite[Chapter 5, Theorem 4.3 (iii)]{Daniel}). We will prove the following results:

	\begin{theorem}
\label{MainResult1}
Let $n\geq 4$ be a natural number and let $G:=SD_{n}$. Then
 \[\Delta(G)=
 \left\{
  \begin{array}{@{}ll@{}}
  2 \,\,\,\mathrm{if}\, \,n=4,\,\\
  
   3   \,\,\,\mathrm{if}\, \,n\geq 5.\,\,\,\,      \\
             \end{array}
 \right. \] 
\end{theorem}

	\begin{theorem}
\label{Result2}
 Let $n\geq 3$ be a natural number and let $G:=Q_{n}$. Then
 \[\Delta(G)=
 \left\{
  \begin{array}{@{}ll@{}}
  2 \,\,\,\mathrm{if}\, \,n=3,\,\\

   3   \,\,\,\mathrm{if}\, \,n\geq 4.\,\,\,\,      \\
             \end{array}
 \right. \] 
\end{theorem}

	\begin{theorem}
\label{Result3}
 Let $p$ be a prime number and $n$ be a natural number, where $n\geq 4$ if $p=2$ and $n\geq 3$ if $p$ is odd. Let $G:=M_n{(p)}$. Then
 \[\Delta(G)=
 \left\{
  \begin{array}{@{}ll@{}}
  2^{n-3}+1 \,\,\,\mathrm{if}\, \,p=2,\,\\

   \frac{p^{n-2}+p-2}{2}   \,\,\,\mathrm{if}\, \,p \,\,\mathrm{is\,\,odd.}\,\,\,      
             \end{array}
 \right.\] 
 
\end{theorem}
	
	With Theorem \ref{FawazResult} and Theorems \ref{MainResult1}\large{-}\ref{Result3}, we have completely calculated the conjugacy diameters of non-abelian finite $p$-groups with cyclic maximal subgroups.

	\section{Preliminaries} 
	\label{preliminaries} 
	
	In this section, we collect some results and fix some notation needed for the proofs of Theorems \ref{MainResult1}-\ref{Result3}

\begin{definition} \label{defbxn}\cite[Section 2]{Ben} 
Let $X$ be a subset of a group $G$. For any $n\geq 0,$ we define $B_X(n)$ to be the set of all elements of $G$ which can be expressed as a product of at most $n$ conjugates of elements of $X$ and their inverses.
\end{definition}

 By Definition \ref{defbxn}, we have \[ \{1\}= B_{X}(0)\subseteq B_{X}(1)\subseteq B_{X}(2)\subseteq \dots\]    
 \vspace{0.5 mm}
  The next result follows from the above definition.
 \begin{lemma}
 \cite[Lemma $2.3\,(iii)$]{Ben} \label{lemma2.3} Let $G$ be a group, let $X, Y\subseteq G$ and $n,m\in \mathbb{N}$.\,Then $B_{X}(n)B_{X}(m) = B_{X}(n + m).$
            \label{lemma2.3:03}      
    \end{lemma} 
\noindent The following terminology is not standard, but it will be useful in this paper.
    \begin{definition} 
	Let $n$ be a natural number. 	
	\begin{enumerate}[(i)]

	\item  If $n\geq 4$ and $G=SD_n,$ then a pair $(a,b)$ of elements $a,b\in G$ is called a standard generator pair if $ord(a)=2^{n-1}, ord(b)=2, bab=a^{2^{n-2}-1}.$  
	
\item  If $n\geq 3$ and $G=Q_n,$ then a pair $(a,b)$ of elements $a,b\in G$ is called a standard generator pair  if $ord(a)=2^{n-1}, b^2 = a^{2^{n-2}}, b^{-1}ab=a^{-1}.$  
	
	\item  If $G=M_n(p),$ where $n\geq 4$ and $p=2$ or $n\geq 3$ and $p$ is an odd prime, then a pair $(a,b)$ of elements $a,b\in G$ is called a standard generator pair if $ord(a)=p^{n-1}, ord(b)=p, b^{-1}ab=a^{p^{n-2}+1}. $

	\end{enumerate}
	\end{definition}

	

	 

	 

                              
	

	\begin{lemma}
	\label{normalsubsinSDN}
Let $G=SD_n$ for some $n\geq 4$ or $G=Q_n$ for some $n\geq 3,$ and let $(a,b)$ be a standard generator pair of $G.$
The subgroups $\langle a\rangle , \langle a^{2}\rangle \langle ab\rangle\,, \langle a^{2}\rangle \langle b\rangle$ are proper normal subgroups of $G.$
	\end{lemma}
	\begin{proof}
	We have 
	
	\begin{equation*}
	|\langle a \rangle|= |\langle a^{2}  \rangle \langle   b \rangle|= |\langle  a^{2}\rangle  \langle ab \rangle|=2^{n-1}.
\end{equation*} 
	
	So, each of these subgroups is a maximal subgroup of $G.$ This proves the lemma since each maximal subgroup of a finite $2$-group is a proper normal subgroup. 
	\end{proof}

\begin{remark}
\label{easycalculation1}
Let $n\geq 4$ be a natural number and $G=SD_n.$ Let $(a,b)$ be a standard generator pair of $G.$
 Let $m\in \mathbb{Z},$ then one can easily show

\[ (a^{-1+2^{n-2}})^{m}=
 \left\{
  \begin{array}{@{}ll@{}}
  a^{-m}  \,\,\,\mathrm{if}\, \,m\,\, \mbox{is even,}\, \,\\
  
   a^{-m} a{^{2}} ^{n-2} \,\,\,\mathrm{if}\, \,m\,\mbox{is odd.}\,\,\,\,\,      \\
             \end{array}
 \right. \]

\end{remark}

\begin{lemma}
\label{conjugacyclassesinSDn}
Let $n\geq 4$ be a natural number and $G=SD_n.$ Let $(a,b)$ be a standard generator pair of $G.$
The following hold.
\begin{enumerate}[(i)]

\item  The conjugacy class of $a^{2}b$ is $\big\{a^{r}b\big| \,r\,\mbox{is even and} \,\,0\leq r< 2^{n-1}\big\}.$
\label{conjugacyclassesinSDn:2}

\item The conjugacy class of $ab$ is $\big\{a^{r}b\big| \,r\,\mbox{is odd and} \,\,0< r< 2^{n-1}\big\}.$
\label{conjugacyclassesinSDn:1}

\item If $0\leq m< 2^{n-1},$ then the conjugacy class of $a^{m}$ is  $\big\{a^{m}, (a^{-1+2^{n-2}})^{m}\big\}.$
\label{conjugacyclassesinSDn:3}

\item If $0\leq m <2^{n-1}$ is even, then $(a^{m}b)^{-1}=a^{{m}^{\prime}}b$ for some even $0\leq m^{\prime} <2^{n-1}.$

\label{conjugacyclassesinSDn:4}

 \item If $0<m <2^{n-1}$ is odd, then $(a^{m}b)^{-1}=a^{{m}^{\prime}}b$ for some odd $0< m^{\prime} <2^{n-1}.$
\label{conjugacyclassesinSDn:5}

\end{enumerate} 
\begin{proof}

 Let $0\leq m < 2^{n-1}$ and $\ell\in \mathbb{Z}.$ For (\ref{conjugacyclassesinSDn:2}) and (\ref{conjugacyclassesinSDn:1}), we have

\begin{align*}
(a^{m}b)^{a^{\ell}} & =a^{-\ell} \, a^{m}b\, a^{\ell}   \\&
=  a^{m-\ell}\,b\,a^{\ell}\,b\,b\\&
=  a^{m-\ell}\,(b\,a\,b)^{\ell}\,b\\&
=  a^{m-\ell}\,(a^{-1+2^{n-2}})^{\ell}\,b\\&
\stackrel{\mathrm{(\ref{easycalculation1})
}}{=}
 \left\{
  \begin{array}{@{}ll@{}}
  a^{m-2\ell}\, b\, ,  \,\,\,\mathrm{if}\, \,\ell\,\, \mbox{is even}\, ,\,\\
   a^{m-2\ell\textcolor{black}{+}2^{n-2}}\, b\, , \,\,\,\mathrm{if}\, \,\ell\,\mbox{is odd}\,.\,\,\,\,      \\
             \end{array}
 \right.  
 \end{align*}
\begin{align*}
(a^{m}b)^{a^{\ell}b} & =(a^{\ell}b)^{-1}  (a^{m}b) (a^{\ell}b)   \\&
=  b\, a^{-\ell}\,a^{m}\,b\,a^{\ell}\, b  \\&
=  (b\, a^{-\ell}\, b)\, (b\, a^{m}\,b)\,a^{\ell}\, b  \\&
=  (b\, a\, b)^{-\ell}\, (b\, a\,b)^{m}\,a^{\ell}\, b  \\&
=  (a^{-1+2^{n-2}})^{-\ell}\, (a^{-1+2^{n-2}})^{m}\,a^{\ell}\, b  \\&
=  a^{\textcolor{black}{2\ell-m+2^{n-2}(m -\ell)}}\,b.\\
 \end{align*}
Thus, if $m$ is even, then the conjugacy class of $a^{m}b$ is $\big\{a^{r}b\big| \,r\,\mbox{is even and} \,\,\textcolor{black}{0\leq r< 2^{n-1}}\big\}.$ In particular,  (\ref{conjugacyclassesinSDn:2}) holds.\,If $m$ is odd, then the conjugacy class of $a^{m}b$ is $\big\{a^{r}b\big| \,r\,\mbox{is odd and} \,\,\textcolor{black}{0< r< 2^{n-1}}\big\}.$ In particular, (\ref{conjugacyclassesinSDn:1}) holds.

For (\ref{conjugacyclassesinSDn:3}), we have 
\begin{align*}
(a^{m})^{a^{\ell}} & =a^{-\ell}\, a^{m}\, a^{\ell}   \\&
=  a^{m}.&
 \end{align*}

\begin{align*}
(a^{m})^{a^{\ell}b} & =(a^{\ell}b)^{-1} \, (a^{m})\, (a^{\ell}b)  \hspace{3.25cm} \\&
=  b\, a^{-\ell}\,a^{m} a^{\ell}b \\&
= b\,\,a^{m}b  \\&
= (b\,\,a\,b)^{m}  \\&
= {(a^{-1+2^{n-2}})}^{m}.
 \end{align*}
This completes the proof of (\ref{conjugacyclassesinSDn:3}).

For (\ref{conjugacyclassesinSDn:4}), we assume that $0\leq m< 2^{n-1}$ is even. We have $a^{m}b\in \langle a^{2} \rangle \langle b \rangle.$ Therefore, we have $(a^{m}b)^{-1}\in \langle a^{2} \rangle \langle b \rangle,$ and $(a^{m}b)^{-1}\notin \langle a^{2} \rangle$ as $a^{m}b\notin \langle a^{2} \rangle.$ It follows that $(a^{m} b)^{-1}= a^{{m}^{\prime}}b$ for some even $0\leq m^{\prime}<2^{n-1}.$

For (\ref{conjugacyclassesinSDn:5}), we assume that $0< m< 2^{n-1}$ is odd. It is clear that $(a^{m}b)^{-1}\notin \langle a \rangle.$ So we have $(a^{m}b)^{-1}= a^{{m}^{\prime}}b$ for some $0\leq m^{\prime}<2^{n-1}.$ Since $a^{m}b\notin \langle a^{2} \rangle \langle b \rangle,$ we have that $m^{\prime}$ is odd.
\end{proof}
\end{lemma}

		\begin{remark}
\label{RemarkQn}
Let $n\geq 3$ be a natural number and $G=Q_n.$ Let $(a,b)$ be a standard generator pair of $G.$
 Let $m\in \mathbb{Z}.$ One can easily show that

\begin{enumerate}[(i)]

\item  $b^{-1} a^{m} b= (b^{-1} a b)^{m}=a^{-m}.$
 
\item $b^{4}=1.$

\end{enumerate}

\end{remark}

\begin{lemma}
\label{conjugacyclassesinQn}
Let $n\geq 3$ be a natural number and $G=Q_n.$ Let $(a,b)$ be a standard generator pair of $G.$ The following hold.
\begin{enumerate}[(i)]
\item  The conjugacy class of $a^{2}b$ is $\big\{a^{r}b\big| \,r\,\mbox{is even and} \,\,0\leq r< 2^{n-1}\big\}.$
\label{conjugacyclassesinQn:2}
\item The conjugacy class of $ab$ is $\big\{a^{r}b\big| \,r\,\mbox{is odd and} \,\,0< r< 2^{n-1}\big\}.$
\label{conjugacyclassesinQn:1}

\item If $0\leq m< 2^{n-1},$ then the conjugacy class of $a^{m}$ is  $\big\{a^{m}, a^{-m}\big\}.$
\label{conjugacyclassesinQn:3}

\item If $0\leq m <2^{n-1}$ is even, then $(a^{m}b)^{-1}=a^{{m}^{\prime}}b$ for some even $0\leq m^{\prime} <2^{n-1}.$

\label{conjugacyclassesinQn:4}

 \item If $0<m <2^{n-1}$ is odd, then $(a^{m}b)^{-1}=a^{{m}^{\prime}}b$ for some odd $0< m^{\prime} <2^{n-1}.$
\label{conjugacyclassesinQn:5}

\end{enumerate} 
\begin{proof}

 Let $0\leq m< 2^{n-1}$ and $\ell\in \mathbb{Z}.$ For (\ref{conjugacyclassesinQn:2}) and (\ref{conjugacyclassesinQn:1}), we have

\begin{align*}
(a^{m}b)^{a^{\ell}} & =a^{-\ell} \, a^{m}b\, a^{\ell}   \hspace{4cm}\\&
=  a^{m-\ell}\,b^{2}\,b^{-1}\,a^{\ell}\,b\,b^{3}\\&
=  a^{m-\ell}\,a^{2^{n-2}}\,(b^{-1}\,a\,b)^{\ell}\,b^{2}\,b\\&
=  a^{m-\ell}\,a^{2^{n-2}}\,a^{-\ell}\,a^{2^{n-2}}\,b\\&
=  a^{m-2\ell}\,b.\\
 \end{align*}
\begin{align*}
(a^{m}b)^{a^{\ell}b} & =(a^{\ell}b)^{-1} \, (a^{m}b)\, (a^{\ell}b)   \hspace{3.2cm} \\&
=  b^{-1}\, a^{-\ell}\,a^{m}b\, a^{\ell}b  \\&
=  b^{-1} (a^{m} b)^{a^{\ell}} b\\&
=  b^{-1} a^{m-2\ell} b\,b\\&
=  a^{2\ell-m}\,b.\\
 \end{align*}

Thus, if $m$ is even, then the conjugacy class of $a^{m}b$ is $\big\{a^{r}b\big| \,r\,\mbox{is even and} \,\,\textcolor{black}{0\leq r< 2^{n-1}}\big\}.$ In particular,  (\ref{conjugacyclassesinQn:2}) holds.\,If $m$ is odd, the conjugacy class of $a^{m}b$ is
$\big\{a^{r}b\big| \,r\,\mbox{is odd and} \,\,\textcolor{black}{0< r< 2^{n-1}}\big\}$.
 In particular,  (\ref{conjugacyclassesinQn:1}) holds. 

For (\ref{conjugacyclassesinQn:3}), we have 
\begin{align*}
(a^{m})^{a^{\ell}} & =a^{-\ell} a^{m} a^{\ell}   \\&
=  a^{m}.&
 \end{align*}
\begin{align*}
(a^{m})^{a^{\ell}b} & =(a^{\ell}b)^{-1} (a^{m}) (a^{\ell}b)   \hspace{3.20cm} \\&
=  b^{-1}\, a^{-\ell}\,a^{m} a^{\ell}b  \\&
= b^{-1}\,\,a^{m}b  \\&
= a^{-m}.
 \end{align*}
This completes the proof of (\ref{conjugacyclassesinQn:3}).

 For (\ref{conjugacyclassesinQn:4}) and (\ref{conjugacyclassesinQn:5}), we have similar arguments as in the proofs of Lemma \ref{conjugacyclassesinSDn} (\ref{conjugacyclassesinSDn:4}) and (\ref{conjugacyclassesinSDn:5}).
\end{proof} 
\end{lemma}

	\begin{lemma}
\label{normallygeneratinfsetinSDn}
Let $G=SD_n$ for some $n\geq 4$ or $G=Q_n$ for some $n\geq 3,$ and let $(a,b)$ be a standard generator pair of $G.$ Let $S$ be a normally generating subset of $G.$ Then there are elements $x,y\in S$ such that $x=a^{\ell}b$ for some $0\leq\ell< 2^{n-1},$ and such that $y=a^{m}$ for some odd $0<m<2^{n-1}$ or $y=a^{m}b$ for some $0\leq m\textcolor{black}{<} 2^{n-1}$ with $\ell \not\equiv m \,\mbox{mod}\,\, 2.$ For each such $x$ and $y,$ we have $\{x,y\}\in \Gamma(G).$

 \begin{proof}
 If $S\subseteq \langle a\rangle,$ then  $\langle\langle S\rangle \rangle \leq \langle a\rangle$ since $\langle a\rangle \trianglelefteq G$ (see Lemma \ref{normalsubsinSDN}), which is a contradiction to $S\in \Gamma(G).$ Thus $S\nsubseteq  \langle a\rangle.$\\
 Let $x\in S\setminus \langle a \rangle.$ Then $x=a^{\ell} b$ for some $0\leq \ell < 2^{n-1}.$ Assume that any element of $S$ has the form $a^{m},$ where $0\leq m< 2^{n-1}$ is even, or $a^{m}b$ with $\ell \equiv m \mbox \,\,\mbox{mod}\,\, 2.$ Then $S\subseteq X_1:=\langle a^{2} \rangle \langle b\rangle$ or $S\subseteq X_2:=\langle a^{2} \rangle \langle ab\rangle.$ Since $X_1$ and $X_2$ are proper normal subgroups of $G$ by Lemma \ref{normalsubsinSDN}, it follows that $\langle\langle S\rangle \rangle \neq G,$ which is a contradiction. Consequently, there is some $y\in S$ such that $y=a^{m},$ where $0<m<2^{n-1}$ is odd, or $y=a^{m}b,$  where $0\leq m< 2^{n-1}$ and $\ell  \not\equiv m\,\, \mbox{mod}\,2.$ \\
  Let  $x,y\in S$ be as above. We show that $X:=\{x,y\}\in \Gamma(G).$ Assume that $y=a^{m}$ for some odd $0<m<2^{n-1}.$ Then $\langle a\rangle = \langle y\rangle\subseteq \langle\langle X\rangle \rangle.$ It follows that $G=\langle a\rangle \langle x\rangle\subseteq \langle\langle X\rangle\rangle.$ Thus $G=\langle\langle X\rangle\rangle,$ and so we have $X\in \Gamma(G).$  Assume now that $y=a^{m}b,$ where $0\leq m< 2^{n-1}$ and $\ell \not\equiv m \,\mbox{mod}\,\, 2.$ Then every element of $S\setminus \langle a\rangle$ is conjugate to $x$ or $y$ by Lemmas \ref{conjugacyclassesinSDn} and \ref{conjugacyclassesinQn}, so $S\setminus \langle a\rangle  \subseteq \langle\langle X\rangle\rangle.$ Also, $a=ab\cdot b\in  \langle\langle X\rangle\rangle$ if $G=SD_n$ and $a=a^{2^{n-2}+1}b\cdot b\in  \langle\langle X\rangle\rangle$ if $G=Q_n.$ It follows that $G= \langle\langle X\rangle\rangle,$ whence $X\in \Gamma(G).$ The proof is now complete.  
   \end{proof}
 \end{lemma}

\begin{lemma}

	\label{UsefulcalculationsSDnn}
	Let $n\geq 4$ be a natural number and $G=SD_n.$ Let $(a,b)$ be a standard generator pair of $G.$
	Let $m_1,m_2\in \mathbb{Z}.$ Then

\begin{enumerate}[(i)]
\item $a^{m_1} b \cdot a^{m_2} b= a^{m_1-m_2+2^{n-2} m_2}.$
\label{UsefulcalculationsSDnn:1}
\item  $a^{m_1} b\cdot a^{m_2}=a^{m_1-m_2+2^{n-2} m_2}\, b.$
\label{UsefulcalculationsSDnn:2}
\item $a^{m_1}\cdot a^{m_2} b=a^{m_1+m_2} b.$
\label{UsefulcalculationsSDnn:3}

\begin{proof}

 For (\ref{UsefulcalculationsSDnn:1}), we have  

\begin{align*}
a^{m_1} b \cdot a^{m_2} b & = a^{m_1} \, (b\, a\, b)^{m_2}  \hspace{4.50cm} \\&
=  a^{m_1} \, (a^{-1+2^{n-2}})^{m_2}  \hspace{4.50cm}\\&
= a^{m_1-m_2+2^{n-2} m_2}.\\&
 \end{align*}

For (\ref{UsefulcalculationsSDnn:2}), we have  

\begin{align*}
a^{m_1} b\cdot a^{m_2} & = a^{m_1} \,b\, a^{m_2}\, b\, b   \hspace{4.60cm}\\&
=   a^{m_1} \,(b\, a\, b)^{m_2}\, b  \hspace{4.60cm}\\&
= a^{m_1} \,(a^{-1+2^{n-2}})^{m_2}\, b  \hspace{4.60cm}\\&
 =a^{m_1-m_2+2^{n-2} m_2}\, b.
 \end{align*}

Statement (\ref{UsefulcalculationsSDnn:3}) is clear. 
\end{proof}

\end{enumerate} 
	
	\end{lemma}

		\begin{lemma}

	\label{UsefulcalculationsQn}
	Let $n\geq 3$ be a natural number and $G=Q_n.$ Let $(a,b)$ be a standard generator pair of $G.$
	Let $m_1,m_2\in \mathbb{Z}.$ Then

\begin{enumerate}[(i)]
\item $a^{m_1} b \cdot a^{m_2} b=a^{m_1-m_2+2^{n-2}}.$
\label{UsefulcalculationsQn:1}
\item  $a^{m_1} b\cdot a^{m_2}=a^{m_1-m_2}b .$
\label{UsefulcalculationsQn:2}
\item $a^{m_1}\cdot a^{m_2} b=a^{m_1+m_2} b.$
\label{UsefulcalculationsQn:3}

\begin{proof}

 For (\ref{UsefulcalculationsQn:1}), we have  

\begin{align*}
a^{m_1} b \cdot a^{m_2} b & = a^{m_1} b \, b \, b^{-1}\, a^{m_2}\, b   \hspace{4.50cm}\\&
=  a^{m_1} b^{2}\, (b^{-1}\, a\, b)^{m_2} \\&
= a^{m_1} a^{2^{n-2}} a^{-m_2}\\&
=a^{m_1-m_2+2^{n-2}}.
 \end{align*}

For (\ref{UsefulcalculationsQn:2}), we have  

\begin{align*}
a^{m_1} b\cdot a^{m_2} &= a^{m_1} b\, b\, b^{-1} a^{m_2}\, b\,b^{3}   \hspace{4.50cm}\\&
=  a^{m_1} b^{2}\, (b^{-1} a\, b)^{m_2}\,b^{2} \, b \\&
= a^{m_1} a^{2^{n-2}}\, a^{-m_2}\,a^{2^{n-2}} \, b\\&
=a^{m_1-m_2} \, b.
 \end{align*}

Statement (\ref{UsefulcalculationsQn:3}) is clear. 
\end{proof}
\end{enumerate}

\end{lemma}

	We now prove some properties of modular $p$-groups. From now on until Lemma \ref{Mnp7-1}, we work under the following hypothesis.
	
	\begin{hypothesis}
Let $p$ be a prime number and $n$ be a natural number, where $n\geq 4$ if $p=2$ and $n\geq 3$ if $p$ is odd. Set $G:=M_n(p),$  let $(a,b)$ be a standard generator pair of $G$ and $z:=a^{p^{n-2}}$.
\end{hypothesis}

	\begin{lemma}
	 \label{uniquemaximalofmnp}
	 
$\Phi(G)=Z(G)$ is the unique maximal subgroup of $\langle a\rangle.$ In particular, we have $z\in Z(G).$
\begin{proof}

This follows from \cite[Chapter 5, Theorem 4.3 (i)]{Daniel}.
\end{proof}
 \end{lemma}

	\begin{lemma}
	\label{usefuleformnp1}

	For each $k \in \mathbb{Z}$, we have $b^{k} a b^{-k}= z^{-k} a$.
	
	\begin{proof}
	We show by induction that, for any non-negative integer $k$, we have $b^{k} a b^{-k}= z^{-k} a$ and $b^{-k} a b^{k}= z^{k} a$. This is clear for $k = 0$. 
	
	We now consider the case $k = 1$. We have $b^{-1}ab = a^{p^{n-2}+1} = za$. Since $z \in Z(G)$ by Lemma \ref{uniquemaximalofmnp}, we also have
	\begin{align*}
		a &= b ( b^{-1} a b)b^{-1} \\&
		=  bzab^{-1} \\&
		= zbab^{-1}
	\end{align*}
	 and thus $bab^{-1}=z^{-1}a.$ 
	
	Assume now that $k \ge 2$ is an integer such that $b^{k-1} a b^{-(k-1)}= z^{-(k-1)} a$ and $b^{-(k-1)} a b^{k-1}= z^{k-1} a$. Bearing in mind that $z \in Z(G)$ by Lemma \ref{uniquemaximalofmnp} and that $bab^{-1}=z^{-1}a$, we see that
	\begin{align*}
		b^{k} a b^{-k}  &= b ( b^{k-1} a b^{-(k-1)})b^{-1} \\&
		=  bz^{-(k-1)}ab^{-1} \\&
		= z^{-(k-1)}bab^{-1}\\&
		=z^{-(k-1)} z^{-1} a \\&
		=z^{-k} a.
	\end{align*}

Similarly, we have 

\begin{align*}
	b^{-k} a b^{k}  &= b^{-1} ( b^{-(k-1)} a b^{k-1})b \\&
	=  b^{-1}z^{k-1}ab \\&
	= z^{k-1}b^{-1}ab\\&
	=z^{k} a.
\end{align*}
\end{proof}

\begin{lemma}
	
	\label{misslemma31au}
	Let $\ell, j \in \mathbb{Z}$. For each positive integer $k,$ we have 
	\begin{equation}
	\label{equationformnp}
(a^{\ell} b^{j})^{k}= a^{k\ell} b^{kj} z^{-k \frac{(k-1)}{2} j\ell} .
\end{equation}
	\begin{proof}

	We prove the lemma by induction over $k$. \\
	For $k=1,$ we have 
	\begin{equation*}
a^{k\ell} b^{kj} z^{-k \frac{(k-1)}{2} j\ell}  = 
 a^{\ell} b^{j}
= (a^{\ell} b^{j})^{k}.
\end{equation*}
	
	Suppose now that $k$ is a positive integer for which (\ref{equationformnp}) holds. Then

	\begin{align*}
(a^{\ell} b^{j})^{k+1}  &= a^{\ell} b^{j} (a^{\ell} b^{j})^{k} \\&
=  a^{\ell} b^{j} a^{k\ell} b^{kj}  z^{-k \frac{(k-1)}{2} j\ell}   \\&
= a^{\ell} b^{j} a^{k\ell} b^{-j}  b^{(k+1)j }  z^{-k \frac{(k-1)}{2} j\ell} \\&
= a^{\ell} (b^{j} a b^{-j})^{k\ell}  b^{(k+1)j }  z^{-k \frac{(k-1)}{2} j\ell} \\&
\stackrel{ \ref{usefuleformnp1}{}}= a^{\ell} (z^{-j}a)^{k\ell} b^{(k+1)j} z^{-k \frac{(k-1)}{2} j\ell}  \\&
\stackrel{ \,\,z\in Z(G)} =a^{(k+1)\ell} b^{(k+1)j} z^{-jk\ell-k \frac{(k-1)}{2} j\ell} \\&
= a^{(k+1)\ell} b^{(k+1)j} z^{(-k- \frac{k(k-1)}{2}) j\ell}\\&
=a^{(k+1)\ell} b^{(k+1)j} z^{(\frac{-(k+1)k}{2}) j\ell}.
 \end{align*}
\end{proof}
\end{lemma}

	\begin{lemma}
	
	\label{orderpn-1inmnp}
	Let $S\in \Gamma(G).$ Then $S$ contains an element of order $p^{n-1}.$ 
	
	\begin{proof}

	 Since $S\not\subseteq \langle a^p,b \rangle = \langle a^p \rangle \langle b \rangle \trianglelefteq G$, the set $S$ contains an element of the form $a^{\ell}b^j$, where $0 < \ell < p^{n-1}$ is not divisible by $p$ and where $0 \le j < p$.
	 By Lemma \ref{misslemma31au}, we have 
	 \begin{equation*}
(a^{\ell} b^{j})^{p}= a^{p\ell} b^{pj} z^{-p \frac{(p-1)}{2} j\ell}=(a^{\ell})^{p} z^{-p \frac{(p-1)}{2}j\ell}.
\end{equation*}
If $p=2,$ then we have   $(a^{\ell} b^{j})^{2}= (a^{\ell})^{2}$ or $(a^{\ell})^{2} z.$ In either case, $(a^{\ell} b^{j})^{2}$ has order $2^{n-2},$ which implies that $a^{\ell} b^{j}$ has order $2^{n-1}.$ If $p$ is odd, then  $(a^{\ell} b^{j})^{p}= (a^{\ell})^{p}$ has order $p^{n-2},$ and so $a^{\ell} b^{j}$ has order $p^{n-1}.$ 
	\end{proof}
	
	\end{lemma}
	
\end{lemma}

	\begin{lemma}
	
	\label{lemmaformnpuse1230}
Let $x$ be an element of $G$ with order $p^{n-1}.$ Then there is some $y\in G$ such that $(x,y)$ is a standard generator pair of $G.$ 	
	\begin{proof}
	Since $x$ has order $p^{n-1}$ and $G$ has order $p^{n},$ the subgroup $\langle  x\rangle$ is maximal in $G.$ By Lemma \ref{uniquemaximalofmnp}, we have $\langle  z\rangle\subseteq Z(G)=\Phi(G)\subseteq \langle  x\rangle.$ As $\langle  x\rangle$ is cyclic, $\langle  x\rangle$ has only one subgroup of order $p.$ Hence, $\langle  z\rangle$ is the only subgroup of $\langle  x\rangle$ with order $p.$ Therefore, $\langle  b\rangle$ is not contained in $\langle  x\rangle.$ Hence, $\langle  b\rangle$ is a complement of $\langle  x\rangle$ in $G.$ Now, \cite[Theorem 5.3.2]{HansKurzweil} implies that there is some $y\in \langle  b\rangle$ such that $(x,y)$ is a standard generator pair of $G.$ 
\end{proof}
\end{lemma}

		\begin{lemma}
	\label{standardgeneratorofmnp}
	Let $S\in \Gamma(G).$ Then there is a standard generator pair $(x,y$) of $G$ such that $x\in S$ and $x^{\ell} y^{j}\in S$ for some $0\leq \ell<p^{n-1}$ and some $0< j< p.$

 	\begin{proof}
	
	By Lemma \ref{orderpn-1inmnp}, $S$ contains an element $x$ of order $p^{n-1}.$ By Lemma \ref{lemmaformnpuse1230}, there is an element $y$ of $G$ such that $(x,y)$ is a standard generator pair of $G.$ Since $S\not\subseteq \langle  x\rangle,$ there exist $0\leq \ell < p^{n-1}$ and  $0< j< p$ such that $x^{\ell} y^{j}\in S,$ as required. 
\end{proof}
	
\end{lemma}

		 Given a group $X$ and an element $h$ of $X$, we write $\mathrm{Conj}_X(h)$ for the conjugacy class of $h$ in $X$.
	 
	 \begin{lemma}
	\label{Mnp7-1} 
	 Let $0\leq \ell < p^{n-1} $ and $0<j<p.$ Then  
	 
	 \begin{enumerate}[(i)]
	
\item  $\mbox{Conj}_{G}(a)=\{ z^{r} a \big| 0\leq r< p\}.$
\label{Mnp7-1:1} 
\item  $\mbox{Conj}_{G}(a^{\ell} b^{j})=\{ z^{r} a^{\ell} b^{j} \big| 0\leq r< p\}.$

\label{Mnp7-1:2} 

\item  $\mbox{Conj}_{G}(a^{-1})=\{ z^{r} a^{-1} \big| 0\leq r< p\}.$

\label{Mnp7-1:3} 

\item  $\mbox{Conj}_{G}((a^{\ell} b^{j})^{-1})=\{ z^{r} a^{-\ell} b^{-j} \big| 0\leq r< p\}.$

\label{Mnp7-1:4} 
\end{enumerate}
	 	 \begin{proof}
	
	(\ref{Mnp7-1:1}) We have $a^{b}=za.$ By induction, we conclude that $a^{{b}^{r}}=z^{r} a$ for all $r\geq 0.$ It follows that \[\mbox{Conj}_{G}(a)= \{ {a^{b}}^{r} \big| 0\leq r< p\} =\{ z^{r} a \big| 0\leq r< p\}.\]
	
	(\ref{Mnp7-1:2}) We have $b^{a}=a^{-1} b a=a^{-1} b a b^{-1} b=a^{-1} a^{b^{p-1}} b=a^{-1}z^{p-1}ab = a^{-1} z^{-1} ab=z^{-1} b.$ By induction, we conclude that $b^{{a}^{s}}=z^{-s} b$ for all $s \ge 0$. It follows that 
	\begin{align*}
\mbox{Conj}_{G}(a^{\ell} b^{j})=\{ (a^{\ell} b^{j})^{a^{s} b^{r}} \big| 0\leq s< p^{n-1}, 0\leq r<p\}  \\
\hspace{4cm}\,\,\,\,\,\,\,\,\,\,=  \{ (a^{{\ell})^{b^{r}}}  (b^{j})^{a^{s} b^{r} } \big| 0\leq s< p^{n-1}, 0\leq r<p\}  \\
=\{ (z^{r} a)^{\ell} (z^{-s} b)^{j} \big| 0\leq s< p^{n-1}, 0\leq r<p\} \\
= \{ z^{r\ell-sj} a^{\ell} b^{j} \big| 0\leq s< p^{n-1}, 0\leq r<p\} \\
= \{ z^{r} a^{\ell} b^{j} \big|  0\leq r<p\}. \,\,\,\,\,\,\,\,\,\,\,\,\,\,\,\,\,\,\,\,\,\,\,\,\,\,\,\,\,\,\,\,\,\,\,\,\,\,\,\,\,\,\,\\
 \end{align*}
	(\ref{Mnp7-1:3}) We have 
\begin{align*}
\mbox{Conj}_{G}(a^{-1})=\{ (a^{-1})^{g} \big|g\in G\}  \\
=\{ (a^{g})^{-1} \big|g\in G\} \\
\stackrel{\mathrm{(\ref{Mnp7-1:1}){}}}=\{ (z^{r} a)^{-1} \big| 0\leq r <p\} \\
= \{ z^{r} a^{-1}  \big|  0\leq r<p\}.\\
 \end{align*}
(\ref{Mnp7-1:4}) We have  \[z^{-j\ell} a^{-\ell} b^{-j} a^{\ell} b^{j}= z^{-j\ell} a^{-\ell} (z^{j} a)^{\ell}= z^{-j\ell} a^{-\ell} z^{j\ell} a^{\ell}=1\]
	
	and hence \[  (a^{\ell} {b^{j})}^{-1} = z^{-j\ell} a^{-\ell} b^{-j}.\]

From (\ref{Mnp7-1:2}), we now see that 
\begin{align*}
\mbox{Conj}_{G}((a^{\ell} b^{j})^{-1})=\{ z^{-r} (a^{\ell} b^{j})^{-1} \big|  0\leq r<p\}  \\
\hspace{4cm}\,\,\,\,\,\,\,\,\,\,=  \{ z^{r} a^{-\ell} b^{-j} \big|  0\leq r<p\}.  \\
 \end{align*}
\end{proof}

\end{lemma}

	\begin{lemma}
		 \label{delta2SDn=deltaSDn}
Let $G=SD_n$ for some $n\geq 4$ or $G=Q_n$ for some $n\geq 3$ or $G=M_n(2)$ for some $n\geq 4$ or  $G=M_n(p)$ for some $n\geq 3$ and some odd prime $p$. Then $\Delta_2(G)=\Delta(G).$ 
 \begin{proof}
Since $\Gamma_2(G)\subseteq \Gamma(G)$, we have $\Delta_2(G)\leq\Delta(G).$\,Now, we want to show that $\Delta(G)\leq \Delta_2(G).$ Let $S\in \Gamma(G).$ We see from Lemmas \ref{normallygeneratinfsetinSDn} and \ref{standardgeneratorofmnp} that there is a subset $T$ of $S$ such that $T\in \Gamma_2(G).$  Since $\| G\|_{S}\leq \| G\|_{T},$ we have   
\begin{align*}
 \Delta(G) &= \sup\{ {\| G\|_{S}}|\ S\in \Gamma(G) \}\\
      &\leq  \sup\{ {\| G\|_{T}}|\ T\in \Gamma_2(G) \}\\
      &= \Delta_2(G).    
\end{align*}
\end{proof}
 \end{lemma}

		\section{Proof of Theorem \ref{MainResult1}} 	\label{sectiontoprovesemidihedra}

	\begin{proof}[Proof of Theorem \ref{MainResult1}]  Let $n\geq 4$ be a natural number$,G:=SD_n$  and $(a,b)$ be a standard generator pair of $G.$  Let $0< \hat{o_1}, \hat{o_2} <2^{n-1}$ be fixed and odd, and  $0\leq v_1 < 2^{n-1}$ be fixed and even. Set 
	
	\[ S_1:=\big\{ a^{v_1}b ,a^{\hat{o_1}} \big\},\]
\[ S_2:=\big\{ a^{\hat{o_2}}b ,a^{\hat{o_1}}\big\},\]
\[ S_3:=\big\{a^{\hat{o_2}} b, a^{v_1}b \big\}.\]

	  From Lemmas \ref{normallygeneratinfsetinSDn} and \ref{delta2SDn=deltaSDn}, we see that we only need to consider  $\|G\|_{S_j},$ where $1\leq j\leq 3,$ in order to determine $\Delta(G).$ In the sense of (\ref{definitionofconjG}) and in view of  Lemma \ref{conjugacyclassesinSDn}, we have

	\[ C_1:=\mbox{Conj}_{G}(S_1\textcolor{black}{^{\pm 1}})=\big\{ a^{v }b \big| 0\leq v  < 2^{n-1}  \mbox{ is even} \big\} \cup\big\{ a^{\pm  \hat{o}_1}, (a^{-1+2^{n-2}})^{\pm \hat{o}_1} \big\},\]
\[ C_2:=\mbox{Conj}_{G}(S_2\textcolor{black}{^{\pm 1}})=\big\{ a^{\hat{o}}b \big| 0< \hat{o} < 2^{n-1}  \mbox{ is odd}\big\} \cup \big\{ a^{\pm  \hat{o}_1}, (a^{-1+2^{n-2}})^{\pm \hat{o}_1} \big\} ,\]
\[ C_3:=\mbox{Conj}_{G}(S_3\textcolor{black}{^{\pm 1}})=\big\{ a^{\ell}b \big| 0\leq \ell < 2^{n-1}  \big\}. \]

  For the reader's convenience and to make the proof easy to follow, we set the following:

\[ \textcolor{black}{\hat{C_1}}:= G\setminus C_1= \big\{ a^{\hat{o} }b \big| 0< \hat{o} < 2^{n-1}  \mbox{ is odd} \big\} \cup \langle a \rangle \setminus  \big\{ a^{\pm \hat{o}_1}, (a^{-1+2^{n-2}})^{\pm \hat{o}_1} \big\} ,\]
\[ \textcolor{black}{\hat{C_2}}:= G\setminus C_2= \big\{ a^{v }b  \big| 0\leq v \textcolor{black}{<}    2^{n-1}  \mbox{ is even} \big\}\cup  \langle a \rangle \setminus  \big\{ a^{\pm \hat{o}_1}, (a^{-1+2^{n-2}})^{\pm \hat{o}_1} \big\} ,\] 
\[ \textcolor{black}{\hat{C_3}}:= G\setminus C_3=  \langle a \rangle.\]

	Next,  we study $\|G\|_{S_j},$  where $1\leq j\leq 3.$

(i) We show 
 \[||G||_{S_1}=
 \left\{
  \begin{array}{@{}ll@{}}
  2 \,\,\,\mathrm{if}\, \,n=4,\,\\
  
   3   \,\,\,\mathrm{if}\, \,n\geq 5.\,\,\,\,      \\
             \end{array}
 \right. \]

	  For every $g\in C_1,$ we have  
\[||g||_{S_1}=1.\]
	   Now, suppose that $ g\in  \hat{C}_1\setminus{\{1\}}.$ Then either $g$ is
\[a^{\hat{o} }b= a^{\hat{o_1}} \cdot a^{v }b \in B_{S_{1}}(2),  \]
where $0< \hat{o} < 2^{n-1}$ is odd and $v=\hat{o}-\hat{o}_1,$

or \[g\in \langle a \rangle \setminus  \big\{ a^{\pm \hat{o}_1}, (a^{-1+2^{n-2}})^{\pm \hat{o}_1} \big\}.\] 

If the former holds, then $||g||_{S_1}=2.$ Assume now $g\in \langle a \rangle \setminus  \big\{ a^{\pm \hat{o}_1}, (a^{-1+2^{n-2}})^{\pm \hat{o}_1} \big\}.$ 

Case 1: $n=4.$

It is easy to see that $\big\{ a^{\pm \hat{o}_1}, (a^{-1+2^{n-2}})^{\pm \hat{o}_1} \big\}=\big\{a, a^{3}, a^{5}, a^{7}\big\}.$ Hence, $g$ can be written as

\[a^{v }= a^{v} b \cdot b \in B_{S_1}(2) ,  \]
where $0< v< 8$ is even. Thus, $||g||_{S_1}=2.$	
	
	Case 2: $n\geq 5.$
	
Let $0< \ell < 2^{n-1}$ such that $g=a^{\ell}.$ If $\ell$ is even, then 	
\[ a^{\ell}=a^{\ell}b\cdot b\in  B_{S_1}(2)\]
and hence $||g||_{S_1}=2.$ Assume now that $\ell$ is odd. One can see from Lemma \ref{UsefulcalculationsSDnn} that $||a^{\ell}||_{S_{1}}\neq 2.$ But we have 
\[
a^{\ell}= a^{\ell }b\cdot b \in B_{S_1}(2)\cdot B_{S_1}(1)= B_{S_1}(3)
\]
by Lemma \ref{lemma2.3} and hence $||g||_{S_1}=3.$ \\

Now, let $g=1,$ then by convention $\{1\} = B_{S_{1}}(0)$ and thus $||g||_{S_1}=0.$ Hence,  $||G||_{S_1}=2$ when $n=4$ and  $||G||_{S_1}=3$ when $n\geq 5.$\\

	  (ii)
	  We show 
 \[||G||_{S_2}=
 \left\{
  \begin{array}{@{}ll@{}}
  2 \,\,\,\mathrm{if}\, \,n=4,\,\\
  
   3   \,\,\,\mathrm{if}\, \,n\geq 5.\,\,\,\,      \\
             \end{array}
 \right. \]

	  For every $g\in C_2,$ we have  
\[||g||_{S_2}=1.\]
	   Now, suppose that $ g\in  \hat{C}_2\setminus{\{1\}}.$ Then either $g$ is
\[a^{v }b= a^{\hat{o_1}} \cdot a^{\hat{o}}b \in B_{S_{2}}(2),  \]

where $0\leq v < 2^{n-1}$ is even and $\hat{o}=v-\hat{o}_1,$

or \[g\in \langle a \rangle \setminus  \big\{ a^{\pm \hat{o}_1}, (a^{-1+2^{n-2}})^{\pm \hat{o}_1} \big\}.\] 

If the former holds, then $||g||_{S_2}=2.$ Assume now $g\in \langle a \rangle \setminus  \big\{ a^{\pm \hat{o}_1}, (a^{-1+2^{n-2}})^{\pm \hat{o}_1} \big\}.$ 

Case 1: $n=4.$

As in (i) above, we have  $\big\{ a^{\pm \hat{o}_1}, (a^{-1+2^{n-2}})^{\pm \hat{o}_1} \big\}=\big\{a, a^{3}, a^{5}, a^{7}\big\}.$ Hence, $g$ can be written as 

\[a^{v }= a^{\hat{o}} b \cdot  a b  \stackrel{\mathrm{ \,\,\ref{UsefulcalculationsSDnn}(\ref{UsefulcalculationsSDnn:1})
}}{=} a^{\hat{o}-1+2^{n-2} }\in B_{S_{2}}(2),   \]
	
	 where $0< v< 8$ is even and $\hat{o}=v+1-2^{n-2}.$ Thus, $||g||_{S_2}=2.$	
	
	Case 2: $n\geq 5.$
	
Let $0< \ell < 2^{n-1}$ such that $g=a^{\ell}.$ If $\ell$ is even, then 	
\[ a^{\ell}=a^{\hat{o}} b \cdot  ab  \stackrel{\mathrm{ \,\,\ref{UsefulcalculationsSDnn}(\ref{UsefulcalculationsSDnn:1})
}}{=} a^{\hat{o}-1+2^{n-2} }\in B_{S_{2}}(2),\]
	 
	 where $\hat{o}=\ell+1-2^{n-1}$ and hence $||g||_{S_2}=2.$ Assume now that $\ell$ is odd. One can see from Lemma \ref{UsefulcalculationsSDnn} that $||a^{\ell}||_{S_{2}}\neq 2.$ But we have

\[
a^{\ell}= a^{\ell }b\cdot b \in B_{S_2}(1)\cdot B_{S_2}(2)= B_{S_2}(3)
\]
by Lemma \ref{lemma2.3} and hence $||g||_{S_2}=3.$ \\

Now, let $g=1,$ then by convention $\{1\} = B_{S_{2}}(0)$ and thus $||g||_{S_2}=0.$ Hence,  $||G||_{S_2}=2$ when $n=4$ and  $||G||_{S_2}=3$ when $n\geq 5.$\\

	  (iii) We show $||G||_{S_3}=2.$

	  For every $g\in C_3,$ we have  
	 
	 \[||g||_{S_3}=1.\]
	 
	   Now, suppose that $ g\in  \hat{C}_3\setminus{\{1\}}.$  Then $g$ can be written as
\[a^{\ell}= a^{\ell}b \cdot  b \in B_{S_{3}}(2),  \]
where $0< \ell < 2^{n-1}.$

	Now, let $g=1,$ then by convention $\{1\} = B_{S_{3}}(0)$ and thus $||g||_{S_3}=0.$ So $||G||_{S_{3}}=2.$

	\vspace{1cm}

	In view of (i)-(iii) and Lemma \ref{normallygeneratinfsetinSDn}, we have $\Delta_2(G)=2$ if $n=4$ and  $\Delta_2(G)=3$ if $n>4.$ So the result follows from that fact that $\Delta_2(G)=\Delta(G)$ (see Lemma \ref{delta2SDn=deltaSDn}).
\end{proof} 
	\section{Proof of Theorem \ref{Result2}} 
\label{sectiontoprovequaternion}

	\begin{proof}[Proof of Theorem \ref{Result2}]  Let $n\geq 3$ be a natural number$,G:=Q_n$  and $(a,b)$ be a standard generator pair of $G.$  Let $0< \hat{o_1}, \hat{o_2} <2^{n-1}$ be fixed and odd, and  $0\leq v_1 \,\textcolor{black}{<} \,2^{n-1}$ be fixed and even. Set 
	
	\[ S_1:=\big\{ a^{v_1}b ,a^{\hat{o_1}} \big\},\]
\[ S_2:=\big\{ a^{\hat{o_2}}b ,a^{\hat{o_1}}\big\},\]
\[ S_3:=\big\{a^{\hat{o_2}} b, a^{v_1}b \big\}.\]

	  From Lemmas \ref{normallygeneratinfsetinSDn} and \ref{delta2SDn=deltaSDn}, we see that we only need to consider  $\|G\|_{S_j},$ where $1\leq j\leq 3,$ in order to determine $\Delta(G).$ In the sense of (\ref{definitionofconjG}) and in view of  Lemma \ref{conjugacyclassesinQn}, we have

	\[ \textcolor{black}{C_1}:=\mbox{Conj}_{G}(S_1\textcolor{black}{^{\pm 1}})=\big\{ a^{v }b \big| 0\leq v \textcolor{black}{ <} \,2^{n-1}  \mbox{ is even} \big\} \cup  \big\{ a^{\pm  \hat{o}_1} \big\}, \]
\[ \textcolor{black}{C_2}:=\mbox{Conj}_{G}(S_2\textcolor{black}{^{\pm 1}})=\big\{ a^{\hat{o}}b \big| 0< \hat{o} < 2^{n-1}  \mbox{ is odd}\big\} \cup \{  a^{\pm  \hat{o}_1}  \},\]
\[ \textcolor{black}{C_3}:=\mbox{Conj}_{G}(S_3\textcolor{black}{^{\pm 1}})=\big\{ a^{\ell}b \big| 0\leq \ell < 2^{n-1}  \big\}. \]

	   For the reader's convenience and to make the proof easy to follow, we set the following:

\[ \textcolor{black}{\hat{C_1}}:= G\setminus C_1= \big\{ a^{\hat{o} }b \big| 0< \hat{o} < 2^{n-1}  \mbox{ is odd\big\}} \cup  \langle a \rangle \setminus  \big\{ a^{\pm \hat{o}_1}\big\}, \]
\[ \textcolor{black}{\hat{C_2}}:= G\setminus C_2= \big\{ a^{v }b \big| 0\leq v \textcolor{black}{<}  2^{n-1}  \mbox{ is even \big\}}\cup \langle a \rangle \setminus  \big\{ a^{\pm \hat{o}_1}\big\}, \]
\[ \textcolor{black}{\hat{C_3}}:= G\setminus C_3= \langle a \rangle.\]

	Next,  we study $\|G\|_{S_j},$  where $1\leq j\leq 3.$

(i) We show 
 \[||G||_{S_1}=
 \left\{
  \begin{array}{@{}ll@{}}
  2 \,\,\,\mathrm{if}\, \,n=3,\,\\
  
   3   \,\,\,\mathrm{if}\, \,n\geq 4.\,\,\,\,      \\
             \end{array}
 \right. \]

	  For every $g\in C_1,$ we have  
\[||g||_{S_1}=1.\]
   Now, suppose that $ g\in  \hat{C}_1\setminus{\{1\}}.$ Then either $g$ is

	\[a^{\hat{o} }b= a^{\hat{o_1}} \cdot a^{v }b \in B_{S_{1}}(2),  \]
where $0< \hat{o} < 2^{n-1}$ is odd and $v=\hat{o} -\hat{o}_1,$

	or \[g\in \langle a \rangle \setminus  \big\{ a^{\pm \hat{o}_1} \big\}.\] 

If the former holds, then $||g||_{S_1}=2.$ Assume now $g\in \langle a \rangle \setminus  \big\{ a^{\pm \hat{o}_1}\big\}.$

	Case 1: $n=3.$

It is easy to see that $\big\{ a^{\pm \hat{o}_1} \big\}=\big\{a, a^{3}\big\}.$ Hence, $g$ can be written as

\[a^{2 }= a \cdot a \in B_{S_1}(2). \]

Thus, $||g||_{S_1}=2.$

		Case 2: $n\geq 4.$
	
Let $0< \ell < 2^{n-1}$ such that $g=a^{\ell}.$ If $\ell$ is even, then 	
\textcolor{black}{\[ a^{\ell} \stackrel{\mathrm{ \,\,\ref{UsefulcalculationsQn}(\ref{UsefulcalculationsQn:1})
}}{=} b \cdot a^{2^{n-2} -\ell} b\in B_{S_{1}}(2)\]}
and hence $||g||_{S_1}=2.$ Assume now that $\ell$ is odd. One can see from Lemma \ref{UsefulcalculationsQn} that $||a^{\ell}||_{S_{1}}\neq 2.$ But we have

\textcolor{black}{\[
a^{\ell}= a^{\ell- 2^{n-2}}b\cdot b \in B_{S_1}(2)\cdot B_{S_1}(1)= B_{S_1}(3)
\]}
by Lemma \ref{lemma2.3} and hence $||g||_{S_1}=3.$ \\

Now, let $g=1,$ then by convention $\{1\} = B_{S_{1}}(0)$ and thus $||g||_{S_1}=0.$ Hence,  $||G||_{S_1}=2$ when $n=3$ and  $||G||_{S_1}=3$ when $n\geq 4.$\\

	 	  (ii)
	  We show 
 \[||G||_{S_2}=
 \left\{
  \begin{array}{@{}ll@{}}
  2 \,\,\,\mathrm{if}\, \,n=3,\,\\
  
   3   \,\,\,\mathrm{if}\, \,n\geq 4.\,\,\,\,      \\
             \end{array}
 \right. \]

	  For every $g\in C_2,$ we have  
\[||g||_{S_2}=1.\]
	   Now, suppose that $ g\in  \hat{C}_2\setminus{\{1\}}.$ Then either $g$ is
\[a^{v }b= a^{\hat{o_1}} \cdot a^{\hat{o}}b \in B_{S_{2}}(2),  \]

where $0\leq v < 2^{n-1}$ is even and $\hat{o}=v-\hat{o}_1,$

or \[g\in \langle a \rangle \setminus  \big\{ a^{\pm \hat{o}_1}\big\}.\] 

If the former holds, then $||g||_{S_2}=2.$ Assume now $g\in \langle a \rangle \setminus  \big\{ a^{\pm \hat{o}_1} \big\}.$ 

Case 1: $n=3.$

As in (i) above, we have  $\big\{ a^{\pm \hat{o}_1} \big\}=\big\{a, a^{3}\big\}.$ Hence, $g$ can be written as 

\[a^{2 }= a\cdot a\in B_{S_{2}}(2).   \]

Thus, $||g||_{S_2}=2.$	
	
	Case 2: $n\geq 4.$
	
Let $0< \ell < 2^{n-1}$ such that $g=a^{\ell}.$ If $\ell$ is even, then 	
\[ a^{\ell}=a^{\hat{o}} b \cdot  ab  \stackrel{\mathrm{ \,\,\ref{UsefulcalculationsQn}(\ref{UsefulcalculationsQn:1})
}}{=} a^{\hat{o}-1+2^{n-2} }\in B_{S_{2}}(2),\]
	 
	 where $\hat{o}=\ell+1-2^{n-2}$ and hence $||g||_{S_2}=2.$ Assume now that $\ell$ is odd. One can see from Lemma \ref{UsefulcalculationsQn} that $||a^{\ell}||_{S_{2}}\neq 2.$ But we have

\[
a^{\ell}= a^{\ell-2^{n-2} } b\cdot b \in \textcolor{black}{B_{S_2}(1)\cdot B_{S_2}(2)= B_{S_2}(3)}
\]
by Lemma \ref{lemma2.3} and hence $||g||_{S_2}=3.$ \\

Now, let $g=1,$ then by convention $\{1\} = B_{S_{2}}(0)$ and thus $||g||_{S_2}=0.$ Hence,  $||G||_{S_2}=2$ when $n=3$ and  $||G||_{S_2}=3$ when $n\geq 4.$\\

	  (iii) We show $||G||_{S_3}=2.$

	  For every $g\in C_3,$ we have  
\[||g||_{S_3}=1.\]
	 
	   Now, suppose that $ g\in  \hat{C}_3\setminus{\{1\}}.$  Then $g$ can be written as

	\[a^{\ell}= a^{\ell-2^{n-2}}b \cdot  b \in B_{S_{3}}(2),  \]

	where $0< \ell < 2^{n-1}.$

	Now, let $g=1,$ then by convention $\{1\} = B_{S_{3}}(0)$ and thus $||g||_{S_3}=0.$ So $||G||_{S_{3}}=2.$

	\vspace{1cm}

	In view of (i)-(iii) and Lemma \ref{normallygeneratinfsetinSDn}, we have $\Delta_2(G)=2$ if $n=3$ and  $\Delta_2(G)=3$ if $n\geq 4.$ So the result follows from that fact that $\Delta_2(G)=\Delta(G)$ (see Lemma \ref{delta2SDn=deltaSDn}).
\end{proof}

\section{\textcolor{black}{ Proof of Theorem \ref{Result3}  }}

To establish Theorem \ref{Result3}, we need to prove a series of lemmas first. In what follows, we work under the following hypothesis. 

\begin{hypothesis}
Let $p$ be a prime number and $n$ be a natural number, where $n\geq 4$ if $p=2$ and $n\geq 3$ if $p$ is odd. Set $G:=M_n(p),$  let $(a,b)$ be a standard generator pair of $G$ and $z:=a^{p^{n-2}}$. Moreover, let $S:=\{a, a^{\ell} b^{j}\}$ for some  $0\leq \ell < p^{n-1} $ and some $0<j<p.$
\end{hypothesis}	

Note that $z \in Z(G)$ by Lemma \ref{uniquemaximalofmnp}. 	
	
		\vspace{0.5cm}
	
	\begin{lemma}
	
	\label{namewithoutdate}
	
	Let $g\in \langle a\rangle.$ Then:
	
		 \begin{enumerate}[(i)]
	
\item If $p=2,$ then $\| g\|_{S}\leq 2^{n-3}.$ 
\label{namewithoutdate:1}

\item If $p$ is odd and $(p,n)\neq (3,3),$ then $\| g\|_{S}\leq \frac{p^{n-2}-1}{2}.$

\label{namewithoutdate:2}
\item If $p=3=n,$ then $\| g\|_{S}\leq 2.$
\label{namewithoutdate:3}

\end{enumerate}

	\begin{proof}
	
	Let $0\leq r < p^{n-1}$ with $g=a^{r}.$ Assume that $r$ is divisible by $p^{n-2}$. Then we have \[g=a^{r}= (a^{{p}^{n-2}})^{\frac{r}{p^{n-2}}}=z^{\left(\frac{r}{p^{n-2}}\right)}.\]
	
	So, with $s:=\frac{r}{p^{n-2}},$ we have $g=z^s=a^{-1}\cdot z^{s} a\in B_S(2)$ by Lemma \ref{Mnp7-1}. Thus $\| g\|_{S}\leq 2$, and hence, (i), (ii) and (iii) hold when $r$ is divisible by $p^{n-2}$. 
	
	We assume now that $r$ is not divisible by $p^{n-2}$, and we treat the cases $p = 2$ and $p$ odd separately.
	
	Case 1: $p=2$.
	
	We may write $r$ in the form $r=m\cdot 2^{n-2}+s,$ where $m\in \{0,1,2\}$ and $-2^{n-3}< s \leq 2^{n-3}.$ Note that $s\neq 0$ by choice of $r.$ Then $g=a^{m\cdot 2^{n-2}+s}=(a^{{2}^{n-2}})^{m} a^{s}=z^{m} a^{s}.$
	
	If $s> 0,$ then $g=z^{m} a\cdot a^{s-1}\in B_S(s)$ by Lemma \ref{Mnp7-1} and hence $\| g\|_{S}\leq s \leq 2^{n-3}.$
	
	If $s<0,$ then $g=z^{m} a^{s}=z^{m} a^{-1}  a^{s+1}=z^{m} a^{-1}  (a^{-1})^{(-s-1)} \in B_S(-s)$ by Lemma \ref{Mnp7-1} and hence $\| g\|_{S}\leq -s < 2^{n-3}.$

	Case 2: $p$ is odd. 
	
	We may write $r$ in the form $r=m\cdot p^{n-2}+s,$ where $m\in \{0,1,\dots,p\}$ and $-\frac{p^{n-2}-1}{2} \leq s \leq \frac{p^{n-2}-1}{2}.$ Note that $s\neq 0$ by choice of $r$. Then $g=a^{m\cdot p^{n-2}+s}=z^{m} a^{s}.$   
	
	If $s> 0,$ then $g=z^{m} a\cdot a^{s-1}\in B_S(s)$ by Lemma \ref{Mnp7-1} and hence $\| g\|_{S}\leq \frac{p^{n-2}-1}{2}.$ In particular, if $(p,n)=(3,3),$ then $\| g\|_{S}=1 < 2.$
	
	If $s< 0,$ then $g=z^{m} a^{s}=z^{m} a^{-1} a^{s+1}=z^{m} a^{-1} (a^{-1})^{(-s-1)}\in B_S(-s)$ by Lemma \ref{Mnp7-1} and hence $\| g\|_{S}\leq -s \leq \frac{p^{n-2}-1}{2}.$ In particular, if $(p,n)=(3,3),$ then $\| g\|_{S}=1 < 2.$
\end{proof}
\end{lemma}

	\begin{lemma}

Let $g\in G\setminus \langle a\rangle.$ Then:
\label{27oct2023}

\begin{enumerate}[(i)]
	\item   If $p=2,$ then  $\| g\|_{S}\leq  2^{n-3}+1$.
	\label{27oct2023:1}
	\item   If $p$ is odd and $(n,p)\neq (3,3),$ then $\| g\|_{S}\leq  \frac{p^{n-2}+p-2}{2}$.
	\label{27oct2023:2}
\end{enumerate}
	
	\begin{proof}
		
		Assume that $p=2.$ Since $g\in  G\setminus \langle a\rangle,$ we have $g=a^{r}b$ for some $0\leq r<2^{n-1}.$ By Lemma \ref{namewithoutdate} (i), we have $a^{r-\ell}\in B_S(2^{n-3}).$ Also, we have $a^{\ell}b\in S\subseteq B_S(1).$ Now, Lemma \ref{lemma2.3} implies that $g=a^{r}b=a^{r-\ell} a^{\ell} b\in B_S(2^{n-3}+1).$ In other words, $\| g\|_{S}\leq  2^{n-3}+1,$ and so (\ref {27oct2023:1}) holds.\\
		Assume now that $p$ is odd and $(n,p) \neq (3,3)$. Since $g\in G \setminus \langle a\rangle,$ we have $g=a^{r}b^{s}$ for some $0\leq r<p^{n-1}$ and some $0<s<p.$ As $0<j<p$ and $\mbox{ord}(b)=p,$ we have $\mbox{ord}(b^j) = p$, and this easily implies that 
		\begin{equation*}
		\langle b \rangle = \langle b^{j}\rangle=\left\{b^{jk} \mid -\frac{p-1}{2}\leq k\leq \frac{p-1}{2}\right\}=\left\{b^{jk} \mid 0\leq k\leq \frac{p-1}{2}\right\} \cup \left\{b^{-jk} \mid 0\leq k\leq \frac{p-1}{2}\right\}. 
		\end{equation*}
		Since $1 \neq b^{s} \in \langle b \rangle$, we thus have $b^{s}=b^{jk}$ or $b^{s}=b^{-jk}$ for some $0<k\leq \frac{p-1}{2}.$ By Lemma \ref{misslemma31au}, we have 
		\begin{equation*}
		(a^{\ell} b^{j})^{k}=a^{k\ell} b^{kj} z^{-\frac{k(k-1)}{2} j\ell}
	    \end{equation*}
    	and
    	\begin{equation*}
    	(a^{-\ell} b^{-j})^{k}=a^{-k\ell} b^{-kj} z^{-\frac{k(k-1)}{2} j\ell}.	
    	\end{equation*}
		Suppose that $b^{s}=b^{jk}.$ Then 
		\begin{equation*}
		(a^{\ell} b^{j})^{k}=a^{k\ell}  z^{-\frac{k(k-1)}{2} j\ell} b^{s}.
		\end{equation*}
		Set $u:=a^{k\ell} z^{-\frac{k(k-1)}{2} j\ell},$ then $(a^{\ell} b^{j})^{k}=ub^{s}.$ We have $g=a^{r}b^{s}=a^{r} u^{-1} u b^{s}=a^{r} u^{-1} (a^{\ell} b^{j})^{k}.$ By Lemma \ref{namewithoutdate} (ii), we have $a^{r} u^{-1} \in B_S(\frac{p^{n-2}-1}{2}).$ Moreover, $(a^{\ell} b^{j})^{k}\in B_S(\frac{p-1}{2})$ since $k\leq\frac{p-1}{2}.$ Applying Lemma \ref{lemma2.3}, we conclude that $g\in B_S(\frac{p^{n-2}-1}{2}+\frac{p-1}{2})=B_S(\frac{p^{n-2}+p-2}{2})$ and hence  $\| g\|_{S}\leq  \frac{p^{n-2}+p-2}{2}.$ This completes the proof of (\ref{27oct2023:2}) for the case that $b^{s}=b^{jk}.$ \\
		Assume now that $b^{s}=b^{-jk}.$ Then 
		\begin{equation*}
		(a^{-\ell} b^{-j})^{k}=a^{-k\ell}  z^{-\frac{k(k-1)}{2} j\ell} b^{s}. 
		\end{equation*}
		With $u:=a^{-k\ell}  z^{-\frac{k(k-1)}{2}j\ell},$ we thus have $(a^{-\ell} b^{-j})^{k}=ub^{s}$ and $g=a^{r} b^{s}=a^{r} u^{-1} u b^{s}=a^{r} u^{-1} (a^{-\ell} b^{-j})^{k}.$ As in the case that $b^{s}=b^{jk},$ we see that  $a^{r} u^{-1} \in B_S(\frac{p^{n-2}-1}{2}),$ and from Lemma \ref{Mnp7-1} (\ref{Mnp7-1:4}), we see that $(a^{-\ell} b^{-j})^k\in B_S(k)\subseteq B_S(\frac{p-1}{2}).$ As in the case $b^{s}=b^{jk}$, it follows that $\| g\|_{S}\leq  \frac{p^{n-2}+p-2}{2}.$ The proof is now  complete. 
	\end{proof}
\end{lemma}
	
	\begin{lemma}
	Assume that $n=p=3.$ Let $g\in G\setminus \langle a\rangle.$ Then $\| g\|_{S}\leq 2.$
	\label{30635}
	\begin{proof}
Let $0\leq r<9$ and $0<s<3$ such that $g=a^{r} b^{s}.$ Note that $\langle b^{s}\rangle=\langle b \rangle = \langle b^{j}\rangle$.\\

Case 1: $\ell$ is not divisible by $3.$ \\
	The elements $a^{\ell}, za^{\ell}, z^{2}a^{\ell}, a^{-\ell}, z a^{-\ell}, z^{2} a^{-\ell}$ are easily seen to be mutually distinct, and they are not contained in $\langle z\rangle.$ So it follows that $\langle a\rangle \backslash \langle z\rangle=\{ a^{\ell}, za^{\ell}, z^{2}a^{\ell}, a^{-\ell}, z a^{-\ell}, z^{2} a^{-\ell}\}$. Likewise, one can see that $\langle a\rangle \backslash \langle z\rangle=\{ a, za, z^{2}a, a^{-1}, z a^{-1}, z^{2} a^{-1}\}$. Together with Lemma \ref{Mnp7-1}, it follows that 
	\begin{align*}
	\{ a^{\ell}, za^{\ell}, z^{2}a^{\ell}, a^{-\ell}, z a^{-\ell}, z^{2} a^{-\ell}\}&=\langle a \rangle \setminus \langle z \rangle \\ &= \{ a, za, z^{2}a, a^{-1}, z a^{-1}, z^{2} a^{-1}\}\\ &= \mathrm{Conj}_G(a) \cup \mathrm{Conj}_G(a^{-1}).
	\end{align*}
	 
Case 1.1: $r$ is not divisible by $3, s=j.$ 
	
	Then $a^{r}\notin \langle z\rangle.$ Therefore, by the above observations, we have $a^{r}=z^{m} a^{\ell}$ or $a^{r}=z^{m} a^{-\ell}$ for some $0\leq m<3.$ If the former holds, then $g=z^{m} a^{\ell} b^{j}\in \mbox{Conj}_{G}(a^{\ell} b^{j})$ by Lemma \ref{Mnp7-1} (ii) and hence $\| g\|_{S}=1 < 2.$ If $a^{r}=z^{m} a^{-\ell},$ then $g=z^{m} a^{-\ell} b^{j}=a^{-2\ell} z^{m} a^{\ell} b^{j}.$ Since $a^{-2\ell} \in \langle a\rangle\backslash \langle z\rangle=\mbox{Conj}_{G}(a)\cup  \mbox{Conj}_{G}(a^{-1})$ and $z^{m} a^{\ell} b^{j}\in \mbox{Conj}_{G}(a^{\ell} b^{j})$ by Lemma \ref{Mnp7-1} (ii), it follows that $\| g\|_{S}\leq 2.$\\
	
	Case 1.2: $r$ is not divisible by $3, s\neq j.$ 
	
	Then $b^{s}=b^{-j}.$  Also, $a^{r}=z^{m} a^{\ell}$ or $z^{m} a^{-\ell}$ for some $0\leq m<3.$ If the latter holds, then $g=z^{m} a^{-\ell} b^{-j}\in B_S(1)$ by Lemma \ref{Mnp7-1} (iv) and hence  $\| g\|_{S}=1 < 2.$ If $a^{r}=z^{m} a^{\ell},$ then $g=z^{m} a^{\ell} b^{-j}=a^{2\ell} z^{m} a^{-\ell} b^{-j},$ and since $a^{2\ell} \in \langle a \rangle \setminus \langle z \rangle = \mbox{Conj}_{G}(a)\cup  \mbox{Conj}_{G}(a^{-1})$ and $z^{m} a^{-\ell} b^{-j} \in \mathrm{Conj}_G((a^{\ell}b^j)^{-1})$ by Lemma \ref{Mnp7-1} (iv), it follows that $\| g\|_{S}\leq 2.$\\
	
	Case 1.3: $r$ is divisible by $3.$
	
	Then $g=z^{m} b^{j}$ or $g=z^{m} b^{-j}$ for some $0\leq m<3.$ In the former case, we have $g=a^{-\ell} z^{m} a^{\ell} b^{j},$ and in the latter case, we have  $g=a^{\ell} z^{m} a^{-\ell} b^{-j}.$ We have $a^{\ell}, a^{-\ell} \in \langle a\rangle\backslash \langle z\rangle=\mbox{Conj}_{G}(a)\cup  \mbox{Conj}_{G}(a^{-1})$, and we also have $z^{m} a^{\ell} b^{j}\in  \mbox{Conj}_{G}(a^{\ell} b^{j})$ by Lemma \ref{Mnp7-1} (ii) and $z^{m} a^{-\ell} b^{-j}\in  \mbox{Conj}_{G}((a^{\ell} b^{j})^{-1})$ by Lemma \ref{Mnp7-1} (iv), and so it follows that $\| g\|_{S}\leq 2.$\\

	Case 2: $\ell$ is divisible by $3.$ \\
	Then $a^{\ell}=z^{t}$ for some $0\leq t<3.$ Note that $g=a^{r}b^{j}$ or $g=a^{r} b^{-j}$. Assume that $r$ is divisible by $3$. Then $g = z^mb^j$ or $g = z^m b^{-j}$ for some $0 \leq m < 3$. If the former holds, then $g=z^{m}b^{j}=z^{m-t}z^{t}b^{j}=z^{m-t}a^{\ell}b^{j}\in \mbox{Conj}_{G}(a^{\ell} b^{j})$ by Lemma \ref{Mnp7-1} (ii) and hence $\| g\|_{S} =1 < 2$. Otherwise, if $g = z^m b^{-j}$, then $g=z^{m} b^{-j} =z^{m+t} z^{-t} b^{-j}= z^{m+t} a^{-\ell} b^{-j}\in \mbox{Conj}_{G}((a^{\ell} b^{j})^{-1})$ by Lemma \ref{Mnp7-1} (iv), whence $\| g\|_{S}= 1 < 2.$ Assume now that $r$ is not divisible by $3$. Then $r-1$ or $r+1$ is divisible by $3,$ and thus $a^{r-1}$ or $a^{r+1}$ lies in $\langle z\rangle.$ Hence $a^{r}=z^{m} a$ or $z^{m} a^{-1}$ for some $0\leq m<3.$ If $a^r = z^ma$, then either $g=az^{m}b^{j} = az^{m-t}z^tb^j = az^{m-t}a^{\ell}b^j$ or $g=az^{m}b^{-j} = az^{m+t}z^{-t}b^{-j} = az^{m+t}a^{-\ell}b^{-j}$, and so it follows from Lemma \ref{Mnp7-1} (ii), (iv) that $\| g\|_{S}\leq 2.$ If $a^r = z^ma^{-1}$, then we either have $g=a^{-1} z^{m}b^{j} = a^{-1}z^{m-t}z^tb^j = a^{-1}z^{m-t}a^{\ell}b^j$ or $g = a^{-1} z^{m}b^{-j} = a^{-1}z^{m+t}z^{-t}b^{-j} = a^{-1}z^{m+t}a^{-\ell}b^{-j}$, and again it follows from Lemma \ref{Mnp7-1} (ii), (iv) that $\| g\|_{S}\leq 2.$
	\end{proof} 
\end{lemma} 

\begin{lemma} 
\label{3oxt23}
Let $X_1 := \langle z \rangle \cdot \lbrace a, a^{-1} \rbrace$, $X_2 := \langle z \rangle \cdot \lbrace b, b^{-1} \rbrace$ and $X := X_1 \cup X_2$. Let $g\in G$ and $s,t$ be nonnegative integers such that $g$ can be written in the form $g = x_1 \cdots x_{s+t}$ such that $x_i \in X$ for all $1 \le i \le s+t$, $|\lbrace 1 \le i \le s+t \mid x_i \in X_1 \rbrace| = s$ and $|\lbrace 1 \le i \le s+t \mid x_i \in X_2 \rbrace| = t$. Then $g=z^{r} a^{s_0} b^{t_0}$ for some $0\leq r<p,$ $-s\leq s_0\leq s$ and $-t\leq t_0\leq t.$ 

\begin{proof} We proceed by induction over $k :=s+t.$ If $k=0,$ then $g=1=z^{0} a^{0} b^{0}$, and so the lemma is true for $k=0.$ Assume now that $k \ge 1$, and suppose that the following holds: If $g' \in G$ and $s^{\prime}, t^{\prime}$ are nonnegative integers with $s^{\prime} + t^{\prime} < k$ such that $g^{\prime}$ can be written in the form $g' =x_1^{\prime} \cdots x_{s^{\prime} + t^{\prime}}^{\prime}$, where $x_i^{\prime} \in X$ for all $1 \le i \le s' + t'$, $|\lbrace 1 \le i \le s^{\prime}+t^{\prime} \mid x_i^{\prime} \in X_1 \rbrace| = s^{\prime}$ and $|\lbrace 1 \le i \le s^{\prime}+t^{\prime} \mid x_i^{\prime} \in X_2 \rbrace| = t^{\prime}$, then $g' = z^{r'} a^{s_0^{\prime}} b^{t_0^{\prime}}$ for some $0 \le r' < p$, $-s' < s_0^{\prime} < s'$, $-t' < t_0^{\prime} < t'$ (induction hypothesis).\\
By the hypotheses of the lemma, we have $g = x_1 \cdots x_k$ for some $x_1, \dots, x_k \in X$, where $|\lbrace 1 \le i \le k \mid x_i \in X_1 \rbrace| = s$ and $|\lbrace 1 \le i \le k \mid x_i \in X_2 \rbrace| = t$. Using the induction hypothesis, we are going to prove that $g=z^{r} a^{s_0} b^{t_0}$ for some $0\leq r<p,$ $-s\leq s_0\leq s$ and $-t\leq t_0\leq t.$ Set $h:=x_1\cdots x_{k-1}.$ We split the proof into two cases.

	Case 1: $x_k\in X_1$ \\
	Then $|\lbrace 1 \le i \le k-1 \mid x_i \in X_1 \rbrace| = s-1$ and $|\lbrace 1 \le i \le k-1 \mid x_i \in X_2 \rbrace| = t$. Therefore, by the induction hypothesis, $h=z^{{r}^{\prime}} a^{{s}^{\prime}_0} b^{{t}^{\prime}_0}$ for some $0\leq r^{\prime}<p, -(s-1)\leq s_0^{\prime}\leq s-1, -t\leq t_0^{\prime}\leq t.$ Since $x_k\in X_1,$ we have $x_k=z^{r^{\prime\prime}} a^{\varepsilon}$ for some $0 \le r^{\prime\prime} < p$ and some $\varepsilon\in \{\pm 1\}.$  Thus $g=h\cdot x_k=z^{{r}^{\prime}} a^{{s}^{\prime}_0} b^{{t}^{\prime}_0} z^{r^{\prime\prime}} a^{\epsilon}=z^{{r}^{\prime}+{r}^{\prime\prime}}  a^{{s}^{\prime}_0} b^{{t}^{\prime}_0} a^{\epsilon}.$ From Lemma \ref{usefuleformnp1}, one can easily see that $b^{{t}^{\prime}_0} a^{\epsilon}= z^{{r}^{\prime\prime\prime}} a^{\epsilon} b^{{t}^{\prime}_0},$ where $r^{\prime\prime\prime}\in \{t^{\prime}_0, -t^{\prime}_0\}.$ Then it follows that $g=z^{r^{\prime}+r^{\prime\prime}}  a^{{s}^{\prime}_0} b^{{t}^{\prime}_0} a^{\epsilon}=z^{r^{\prime}+r^{\prime\prime} +r^{\prime\prime\prime}} a^{{s}^{\prime}_0} a^{\epsilon} b^{{t}^{\prime}_0}=z^{r^{\prime}+r^{\prime\prime} +r^{\prime\prime\prime}} a^{{s}^{\prime}_0+\epsilon} b^{{t}^{\prime}_0}.$ Now, let $0\leq r<p$ with $r\equiv r^{\prime}+r^{\prime\prime} +r^{\prime\prime\prime}\,\mbox{mod}\, p,$ and set $s_0:=s^{\prime}_0+\epsilon, t_0:=t^{\prime}_0.$ Then $g=z^{r} a^{{s}_0} b^{{t}_0},$ and we have $-s\leq -(s-1)+\epsilon\leq s_0\leq (s-1)+\epsilon\leq s,  -t\leq t_0\leq t.$ This completes the proof of the lemma for the case $x_k\in X_1.$

	Case 2: $x_k\in X_2$ \\
	Then $|\lbrace 1 \le i \le k-1 \mid x_i \in X_1 \rbrace| = s$ and $|\lbrace 1 \le i \le k-1 \mid x_i \in X_2 \rbrace| = t-1$. Therefore, by the induction hypothesis, $h=z^{{r}^{\prime}} a^{{s}^{\prime}_0} b^{{t}^{\prime}_0}$ for some $0\leq r^{\prime}<p, -s\leq s_0^{\prime}\leq s, -(t-1)\leq t_0^{\prime}\leq t-1.$ Since $x_k\in X_2,$ we have $x_k=z^{r^{\prime\prime}} b^{\varepsilon}$ for some $0\leq r^{\prime\prime}< p$ and some $\varepsilon\in \{\pm 1\}$. Then  $g=hx_k=z^{{r}^{\prime}} a^{{s}^{\prime}_0} b^{{t}^{\prime}_0} x_k=z^{{r}^{\prime}+r^{\prime\prime}} a^{{s}^{\prime}_0} b^{{t}^{\prime}_0+\epsilon}.$ Let $0\leq r<p$  
	with  $r\equiv r^{\prime}+r^{\prime\prime} \,\mbox{mod}\, p, s_0:=s^{\prime}_0$ and $t_0:=t^{\prime}_0+\epsilon.$ Then $g=z^{r} a^{{s}_0} b^{{t}_0},$ and we have $-s\leq s_0\leq s$ and  $-t\leq t_0\leq t.$ This completes the proof of the lemma for the case $x_k\in X_2.$
\end{proof} 
\end{lemma}

	\begin{lemma}

	\label{equlaityresu}
	Let $T := \lbrace a,b \rbrace$. Then:
	\label{exactMn2}

\begin{enumerate}[(i)]
\item  $\| G\|_{T} \geq 2^{n-3}+1$ if $p=2,$ 
\label{exactMn2:1}
\item  $\| G\|_{T} \geq \frac{p^{n-2}+p-2}{2}$ if $p$ is odd.
\label{exactMn2:2}
\end{enumerate} 

\begin{proof} 
Let $X_1 := \langle z \rangle \cdot \lbrace a, a^{-1} \rbrace$, $X_2 := \langle z \rangle \cdot \lbrace b, b^{-1} \rbrace$ and $X := X_1 \cup X_2$. From Lemma \ref{Mnp7-1}, we see that $X = \mathrm{Conj}_G(T^{\pm 1})$.\\
Assume that $p=2$, and set $g:=a^{2^{n-3}}b.$ We show that $\| g\|_{T}\geq  2^{n-3}+1$. Let $k \geq 1$ and $x_1,\dots, x_k\in X$ with $g=x_1\cdots x_k.$ Set $s := |\lbrace 1 \le i \le k \mid x_i \in X_1 \rbrace|$ and $t := |\lbrace 1 \le i \le k \mid x_i \in X_2 \rbrace|$. Then, by Lemma \ref{3oxt23}, we have $a^{2^{n-3}}b=g=z^{r} a^{s_0} b^{t_0}$ for some $r\in\{0,1\}$ and some $-s \leq s_0\leq s$, $-t \leq t_0\leq t.$ This implies that $a^{2^{n-3}}=z^{r} a^{s_0}$ and $b=b^{t_0}.$ If $s< 2^{n-3}$, then $z^{r} a^{s_0}$ has the form $a^{m}, a^{2^{n-1}-m}, a^{2^{n-2}+m}$ or $a^{2^{n-2}-m}$ for some $0\leq m < 2^{n-3},$ whence $z^{r} a^{s_0} \neq a^{2^{n-3}}.$ Thus $s\geq 2^{n-3}.$ Also, $t\geq 1$ because otherwise $b = b^{t_0}= b^0 = 1$. So we have  $k=s+t\geq 2^{n-3}+1.$ Since $x_1, \dots, x_k$ were arbitrarily chosen elements of $X = \mathrm{Conj}_G(T^{\pm 1})$ with $g = x_1 \cdots x_k$, we can now conclude that $\| g\|_{T}\geq  2^{n-3}+1$. In particular, we have $\| G\|_{T}\geq 2^{n-3}+1$, completing the proof of (\ref{exactMn2:1}).\\
Assume now that $p$ is odd, and set $g:=a^{{\frac{p^{n-2}-1}{2}}} b^{{\frac{{p-1}}{2}}}.$ We show that  $\| g\|_{T}\geq \frac{p^{n-2}+p-2}{2}.$ Let $k \geq 1$ and $x_1,\dots\,,x_k\in X$ with $g=x_1\cdots\,x_k.$ Set $s := |\lbrace 1 \le i \le k \mid x_i \in X_1 \rbrace|$ and $t := |\lbrace 1 \le i \le k \mid x_i \in X_2 \rbrace|$. Then, by Lemma \ref{3oxt23}, we have $a^{\frac{p^{n-2}-1}{2}} b^{\frac{p-1}{2}}=g=z^{r} a^{s_0} b^{t_0}$ for some $0\leq r< p$ and some $-s \leq s_0\leq s,$ $-t \leq t_0\leq t.$ This implies that $a^{\frac{p^{n-2}-1}{2}} = z^{r} a^{s_0}$ and $b^{t_0} = b^{\frac{p-1}{2}}$. If $s< \frac{p^{n-2}-1}{2},$ then one can see similarly as in the case $p=2$ that $z^{r}a^{s_0}$ cannot be $a^{\frac{p^{n-2}-1}{2}}.$ So it follows that $s\geq  \frac{p^{n-2}-1}{2}.$ Also, it follows from $b^{t_0}=b^{\frac{p-1}{2}}$ that $t\geq \frac{p-1}{2}.$ Consequently, we have $k=s+t\geq \frac{p^{n-2}+p-2}{2}$. Since $x_1, \dots, x_k$ were arbitrarily chosen elements of $X = \mathrm{Conj}_G(T^{\pm 1})$ with $g = x_1 \cdots x_k$, we can now conclude that $\| g\|_{T}\geq \frac{p^{n-2}+p-2}{2}$. In particular, we have $\| G\|_{T}\geq \frac{p^{n-2}+p-2}{2},$ completing the proof of (\ref{exactMn2:2}).
\end{proof}
\end{lemma}

With the above lemmas at hand, we can now prove Theorem \ref{Result3}.

	\begin{proof}[Proof of Theorem \ref{Result3}]
Let $p$ be a prime number and $n$ be a natural number, where $n \ge 4$ if $p = 2$ and $n \geq 3$ if $p$ is odd, and let $G := M_n(p)$. 
	Let $S\in \Gamma_2(G)$. By Lemma \ref{standardgeneratorofmnp}, there is a standard generator pair $(a,b)$ of $G$ such that $S=\{a,a^{\ell} b^{j}\}$ for some $0\leq \ell< p^{n-1}$ and some $0<j<p.$\\ 
	Assume that $p=2.$ By Lemmas \ref{namewithoutdate} (\ref{namewithoutdate:1}) and \ref{27oct2023} (\ref{27oct2023:1}), we have $\| G\|_{S}\leq 2^{n-3}+1$. Because of Lemma \ref{equlaityresu}, we even have equality when $S = \lbrace a, b \rbrace$. So it follows that $\Delta_2(G) = 2^{n-3}+1$, and Lemma \ref{delta2SDn=deltaSDn} implies that $\Delta(G)=2^{n-3}+1.$\\ 
	Assume now that $p$ is odd and $(n,p)\neq (3,3).$ By Lemmas \ref{namewithoutdate} (\ref{namewithoutdate:2}) and \ref{27oct2023} (\ref{27oct2023:2}) , we have $\| G\|_{S}\leq  \frac{p^{n-2}+p-2}{2}$, and because of Lemma \ref{equlaityresu}, we even have equality when $S = \lbrace a, b \rbrace$. Applying Lemma \ref{delta2SDn=deltaSDn}, we conclude that $\Delta(G)=\Delta_2(G)=\frac{p^{n-2}+p-2}{2}.$\\
	Assume now that $n=p=3$. From Lemmas \ref{namewithoutdate} (\ref{namewithoutdate:3}) and \ref{30635}, we see that $\| G\|_{S}\leq 2,$ and because of Lemma \ref{equlaityresu}, we even have $\| G\|_{S}= 2$ when $S = \lbrace a, b \rbrace$. Applying Lemma \ref{delta2SDn=deltaSDn}, we conclude that $\Delta(G)=\Delta_2(G)=2=\frac{3^{3-2}+3-2}{2}$.
	\end{proof}

	\section{Conclusions}\
	 After the determination of the conjugacy diameters of dihedral groups in \cite[Theorem 6.0.2]{Fawaz} and \cite[Example 2.8]{Ben}, we proved in this work further results about conjugacy diameters of non-abelian finite groups with cyclic maximal subgroups. Namely, we determined the conjugacy diameters of the semidihedral $2$-groups, the generalized quaternion groups and the modular $p$-groups. By that, the conjugacy diameters of non-abelian finite $p$-groups with cyclic maximal subgroups have been completely calculated. We believe that the strategies applied in the proofs of our results could also be used for studying conjugacy diameters of other finite groups. For example, we think that one could proceed similarly as in the proofs of Theorems \ref{MainResult1} and \ref{Result2} to study conjugacy diameters of the generalised dihedral groups.

	Our results also lead to a question concerning the relation between the conjugacy class sizes of a finite group and its conjugacy diameter. We found that the semidihedral $2$-groups and the generalized quaternion groups have ''small'' conjugacy diameters, while the elements not lying in a maximal subgroup have, in relation to the group order, ''large'' conjugacy class sizes. On the other hand, we found that the conjugacy diameters of the modular $p$-groups $M_n(p)$ grow fast as $n$ grows, while the conjugacy classes of $M_n(p)$ are ''small'' (they have at most $p$ elements). In view of this observation, it would be interesting problem to study how the conjugacy diameters of finite groups are influenced by conjugacy class sizes. Note that \cite[Proposition 7.1]{Ben} pays some attention to this question.\\
\\
	
	\textbf{Conflict of interest:} The authors declare no conflicts of interest.



\end{document}